\documentclass[a4paper, 10pt, conference]{ieeeconf}

\IEEEoverridecommandlockouts
\usepackage{times}
\usepackage{amsmath,amssymb}
\usepackage{amsthm}
\usepackage{mathtools} 
\usepackage{caption,subcaption}
\usepackage{tikz,pgfplots}
\pgfplotsset{compat=newest}
\usetikzlibrary{positioning,calc,backgrounds}
\usepackage{algorithm,algpseudocode,xcolor}
\usepackage{xspace}
\usepackage{csquotes}
\usepackage[hidelinks]{hyperref} 

\DeclareMathAlphabet{\mathpzc}{OT1}{pzc}{m}{it}

\title{\LARGE \bf
    On Data-Driven Unbiased 
    Predictors using the Koopman Operator
}

\author{Roland Schurig$^{1}$, 
        Pieter van Goor$^2$, 
        Karl Worthmann$^3$, 
        Rolf Findeisen$^{1}$%
\thanks{$^{1}$ Control and Cyber-Physical Systems Laboratory, TU Darmstadt, Germany, 
        \{roland.schurig, rolf.findeisen\}@iat.tu-darmstadt.de;  $^{2}$ School of Aerospace, Mechanical, and Mechatronic Engineering, The University of Sydney, Australia,  pieter.vangoor@sydney.edu.au;  $^{3}$ Optimization-based Control Group, Institute of Mathematics, Technische Universit{\"a}t Ilmenau, Germany, karl.worthmann@tu-ilmenau.de}%
}

\newcommand*{\ie}{i.e.,\@\xspace}
\newcommand*{\eg}{e.g.,\@\xspace}

\newcommand*{\linspace}[1]{\ensuremath{ #1 }} 	
\DeclareMathOperator{\ran}{ran} 
\DeclareMathOperator{\tr}{tr} 
\newcommand*{\Matrix}[1]{\ensuremath{ \begin{pmatrix} #1 \end{pmatrix} }} 
\newcommand*{\rmspace}[2]{\ensuremath{ \mathrm{M} (#1 \times #2 , \mathbb{R}) }} 
\newcommand*{\iprod}[2]{\ensuremath{ \langle #1, #2 \rangle }} 	

\newcommand*{\exv}[1]{\ensuremath{ \mathbb{E}(#1) }} 
\newcommand*{\var}[1]{\ensuremath{ \mathrm{var}(#1) }} 

\definecolor{myblue}{HTML}{1F77B4}
\definecolor{myorange}{HTML}{FF7F0E}
\definecolor{mygreen}{HTML}{2CA02C}
\definecolor{myred}{HTML}{D62728}

\newtheorem{lemma}{Lemma}
\newtheorem{proposition}{Proposition}
\newtheorem{corollary}{Corollary}

\theoremstyle{remark}
\newtheorem{remark}{Remark}

\newcommand{\figref}[1]{\hyperref[#1]{Fig.~\ref*{#1}}}
\newcommand{\propref}[1]{\hyperref[#1]{Proposition~\ref*{#1}}}
\newcommand{\lemmaref}[1]{\hyperref[#1]{Lemma~\ref*{#1}}}
\newcommand{\cororef}[1]{\hyperref[#1]{Corollary~\ref*{#1}}}
\newcommand{\secref}[1]{\hyperref[#1]{Section~\ref*{#1}}}

\begin{document}
\twocolumn
\maketitle
\thispagestyle{empty}
\pagestyle{empty}

\begin{abstract}
The Koopman operator and its data-driven approximations, such as extended dynamic mode decomposition (EDMD), are widely used for analysing, modelling, and controlling nonlinear dynamical systems. However, when the true Koopman eigenfunctions cannot be identified from finite data, multi-step predictions may suffer from structural inaccuracies and systematic bias.
To address this issue, we analyse the first and second moments of the multi-step prediction residual. By decomposing the residual into contributions from the one-step approximation error and the propagation of accumulated inaccuracies, we derive a closed-form expression characterising these effects. This analysis enables the development of a novel and computationally efficient algorithm that enforces unbiasedness and reduces variance in the resulting predictor.
The proposed method is validated in numerical simulations, showing improved uncertainty properties compared to standard EDMD. These results lay the foundation for uncertainty-aware and unbiased Koopman-based prediction frameworks that can be extended to controlled and stochastic systems.
\end{abstract}

\section{Introduction}\label{sec:intro}
Data-driven analysis, modelling, and control of dynamical systems remains a challenging task, particularly in the presence of nonlinearities. Hence, many techniques based on \textit{lifting} the dynamics into a higher-dimensional space have been proposed to model and analyse their evolution and structural properties using tools from linear systems theory. A particularly successful method in this realm is extended dynamic mode decomposition~\cite{Williams2015} in its various algorithmic realisations, see, \eg \cite{colbrook2024multiverse}. The theoretical backbone for its analysis is the Koopman operator, which has established itself as a key tool since the seminal papers~\cite{mezic2005spectral,rowley2009spectral}, see also the recent overview~\cite{StraWort26}.
Nowadays, an error analysis is available both in the infinite-dictionary (projection) and infinite-data (estimation) limit~\cite{korda2018convergence}, as well as for settings with finitely many samples; see, e.g., \cite{philipp2024variance} and~\cite{kohne2025error} for recent results for EDMD and kernel EDMD, respectively.

Recent developments focus on reliable predictions in the Koopman framework, addressing inconsistencies under repeated prediction. In \cite{vanGoor2025}, a \emph{maximum-likelihood reprojection} method is proposed to maintain consistency with the Koopman manifold, interpreting reprojection as a statistically optimal operation under Gaussian residuals. The work \cite{schurig2025} introduced a geometric framework for dictionary learning on the Grassmann manifold, optimising observable subspaces for approximate Koopman invariance.

This work focuses on the number of steps that can be reliably predicted by the lifted system in an auto-regressive manner. A probabilistic analysis of the first and second moments of the residual error in the lifted space is conducted in the case of finite dictionary and finite data. This analysis is carried out for a generic predictor in the lifted space, and it is argued that the predictor stemming from EDMD is optimal in the sense that it is unbiased and variance-optimal for one step.

The analysis features an analytic closed-form expression of the expected residual, allowing us to extend this result to longer prediction horizons. It demonstrates a structured way to enforce unbiasedness of the multi-step predictor by zeroing the expectation of the associated residual with respect to the empirical measure induced by the collected data points. This leads to explicit conditions under which the multi-step residual vanishes in expectation and clarifies the remaining degrees of freedom.

The evolution of the overall residual is then expressed in terms of a linear system and local residuals, which can be influenced through the selection of suitable predictors. We derive a basis in the space of linear predictors that makes the conditions for unbiasedness transparent. The remaining degrees of freedom are then leveraged to reduce the multi-step variance. We put special focus on numerical applicability: the proposed predictor has a closed-form solution that can be computed efficiently using tools from numerical linear algebra.

The results presented in this work provide a principled understanding of how multi-step prediction errors accumulate in Koopman-based models and offer insights under which a lifted predictor remains unbiased over multiple time steps. This yields a practical and computationally efficient method for constructing unbiased and variance-reduced multi-step predictors directly from finite data. As such, the proposed framework strengthens the theoretical foundations of data-driven Koopman prediction and enables more reliable long-horizon forecasting, which is essential for applications in control, simulation, and uncertainty-aware decision making.

The paper is structured as follows.
In \secref{sec:koopman_compression}, the Koopman framework is introduced and the EDMD algorithm is presented as a data-driven approximation of the Koopman operator. This leads to the probabilistic analysis of the lifted dynamics in \secref{sec:one_step_predictions}, focusing on the one-step behaviour optimised by EDMD. Based on that, we present necessary and sufficient conditions for the unbiasedness of the predictor over multiple time steps in \secref{sec:one_step_predictions}, which is not guaranteed by the EDMD predictor. Additionally, a possible way to exploit the remaining degrees of freedom for minimising variance is outlined. The proposed predictor is introduced in \secref{sec:predictor}, and a numerical example is presented in \secref{sec:example}. We conclude in \secref{sec:conclusion}.

\paragraph*{Notation}
The notation is standard. 
The natural numbers are denoted by $\mathbb{N}$, and $\mathbb{N}_0 \coloneqq \mathbb{N} \cup \{0\}$.
The Euclidean space of all $(m \times n)$-matrices over $\mathbb{R}$ is denoted by $\rmspace{m}{n}$, with the standard vector space structure and inner product $\iprod{A}{B} \coloneqq \tr(A^\top B)$. 
This induces the Frobenius norm $\| A \|_\mathrm{F} \coloneqq \sqrt{\iprod{A}{A}}$. 
The expectation and variance of a random variable $X$ are denoted by $\exv{X}$ and $\var{X}$, respectively.

\section{Compressions in the Koopman Framework}
\label{sec:koopman_compression}

We consider the discrete-time nonlinear system
\begin{align} \label{eq:nonlinear_sys}
  x(t+1) &= f(x(t)), &
  x(0) = \bar{x},
\end{align}
with continuous state-transition map $f \colon \mathbb{R}^n \to \mathbb{R}^n$. 
For ease of exposition, we assume that $X \subset \mathbb{R}^n$ is non-empty, compact, and forward invariant under $f$, \ie $x \in X$ implies $f(x) \in X$. 
The trajectory starting from the initial state $\bar{x} \in X$ is denoted by $x(\cdot, \bar{x}) \colon \mathbb{N}_0 \to X$.

The linear \emph{Koopman operator} $K$ maps a function $\varphi \in L^2(X,\mathbb{R})$ to another function $K \varphi \in L^2(X,\mathbb{R})$ and is defined by the identity
\begin{equation*}
    (K \varphi)(x) = (\varphi \circ f)(x) \qquad\forall\, x \in X.
\end{equation*}
The functions $\varphi \in L^2(X,\mathbb{R})$ are called \emph{observables}. 
The Koopman operator offers a way to view the dynamical system \eqref{eq:nonlinear_sys} through a linear lens; however, it is infinite-dimensional. 

To obtain a finite-dimensional approximation of $K$, consider $m$ linearly independent and continuous observables $\psi_1, \dots, \psi_m \in L^2(X,\mathbb{R})$, and denote by $\linspace{V} \subset L^2(X,\mathbb{R})$ their span. 
This subspace is finite-dimensional, but $\psi \in \linspace{V}$ does not necessarily imply that the propagated observable $K \psi$ is also contained in $\linspace{V}$. 
Exceptions are so-called Koopman eigenfunctions and invariant subspaces contained in~$V$, which have proven to be extremely useful, see, e.g.,  \cite{mezic2013analysis,mauroy2016global,haseli2021learning}, but may not be observable if, e.g., observables correspond to sensors or output of the plant in consideration.
To mitigate problems resulting from a missing Koopman invariance of the subspace~$V$ in the prediction of dynamical systems, let $P$ be a projection of $L^2(X,\mathbb{R})$ with $\ran P = \linspace{V}$. 
Then, the linear map $C \colon \linspace{V} \to \linspace{V}$ given by
\begin{equation*}
    C \psi \coloneqq P K \psi
\end{equation*}
defines a \emph{compression of $K$ on $\linspace{V}$}, which can be represented by an $(m \times m)$-matrix. If the subspace~$\linspace{V}$ is \emph{Koopman invariant}, i.e., the inclusion $K \psi \in \linspace{V}$ holds for all $\psi \in \linspace{V}$, the compression~$C$ corresponds to the restriction $\left. K \right|_{\linspace{V}}$ of $K$ to $\linspace{V}$. 
Otherwise, projection error is present, which stems from neglecting the component of $K \psi$ that is not in $\linspace{V}$.

The term \emph{Extended Dynamic Mode Decomposition} (EDMD) refers to a data-driven approach to compute the matrix representation of the compressed Koopman operator~$C$. 
First, the \emph{lift} $\Psi \colon X \to \mathbb{R}^m$ associated with the basis $(\psi_1, \dots, \psi_m)$ is defined to be
\begin{equation*}
    \Psi(x) \coloneqq \Matrix{\psi_1(x) & \dots & \psi_m(x)}^\top,
\end{equation*}
which we assume to be injective (one-to-one).
Next, consider the sampled states
\begin{equation} \label{eq:edmd_data}
    x_1, \dots, x_L \in X
\end{equation}
with $L \geq m$, and where the states $x_i$ are i.i.d.\ realisations of a random variable. 
Then, using this data set, define the bilinear form $b \colon L^2(X,\mathbb{R}) \times L^2(X,\mathbb{R}) \to \mathbb{R}$ by
\begin{equation*}
    b(\psi,\varphi) \coloneqq \frac{1}{L} \sum_{i=1}^L \psi(x_i) \varphi(x_i) .
\end{equation*}
Additionally, for a fixed time horizon $N \in \mathbb{N}$, we assume that we have, for each $i \in \{1,\ldots,L\}$, measured the state trajectories $x(t;x_i)$, $t \in \{0,1,\ldots,N\}$.
For each $t \in \{0, \dots, N\}$, the corresponding \emph{data matrix} is then defined to be
\begin{equation*}
    G_t \coloneqq \Matrix{\Psi(x(t,x_1)) & \dots & \Psi(x(t,x_L))} .
\end{equation*}
The EDMD algorithm in its standard form only uses the matrices $G_0$ and $G_1$, \ie information about the one-step behaviour of the dynamical system \eqref{eq:nonlinear_sys}%
\footnote{We comment on the EDMD algorithm that uses all matrices $G_t$ in \secref{sec:predictor}.}.
We assume that $G_0$ has full row rank and thus $b$ is non-degenerate on $\linspace{V}$, meaning the projection $P$ of $L^2(X,\mathbb{R})$ can be defined by the \enquote{orthogonality} property
\begin{equation*}
    \forall \varphi \in L^2(X,\mathbb{R}) \colon \forall \psi \in \linspace{V} \colon 
    b(\varphi - P \varphi, \psi) = 0 .
\end{equation*}
Hence, the projection is not orthogonal with respect to the usual inner product on $L^2(X,\mathbb{R})$, but with respect to the data-induced bilinear form $b$.
The matrix representation $\hat{K}$ of $C$ relative to the basis $(\psi_1, \dots, \psi_m)$ is
\begin{equation*}
    \hat{K} = (G_0 G_0^\top)^{-1} G_0 G_1^\top .
\end{equation*}
This representation has the alternative interpretation of being the unique minimiser to the least-squares problem
\begin{equation*}
    \min_{\tilde{K}} \left\{ \| G_1 - \tilde{K}^\top G_0 \|_\mathrm{F}^2 \right\} .
\end{equation*}
We refer to \cite{Williams2015,colbrook2024multiverse} for general references on EDMD and its variants. 

The compression can be used to propagate the state of~\eqref{eq:nonlinear_sys} in the lifted space $M \coloneqq \operatorname{im} \Psi$.
For each initial state $\bar{x}~\in~ X$, one can apply the Koopman operator component-wise to approximate the successor state $\bar{y} = f(\bar{x})$ with
\begin{equation*}
    \begin{aligned}
        \Psi(\bar{y}) = (\Psi \circ f)(\bar{x}) 
        &= ((K \psi_1)(\bar{x}), \dots, (K \psi_m)(\bar{x})) \\
        &\approx ((C \psi_1)(\bar{x}), \dots, (C \psi_m)(\bar{x})) \\
        &= \hat{K}^\top \Psi(\bar{x}),
\end{aligned}
\end{equation*}
meaning that $\hat{K}^\top \Psi(\bar{x})$ is an estimate of $\Psi(f(\bar{x}))$. 
Note that, while $\bar{y}$ can be recovered from $\Psi(\bar{y})$ since $\Psi$ is assumed to be injective, this may not be possible for $\hat{K}^\top \Psi(\bar{x})$.
The replacement of the Koopman operator $K$ with its compression $C$ introduces the aforementioned projection error, so that there may \emph{not} exist $\hat{x} \in X$ with $\Psi(\hat{x}) = \hat{K}^\top \Psi(\bar{x})$, \ie we have $\hat{K}^\top \Psi(\bar{x}) \notin M$. 
Therefore, before applying $\Psi^{-1}$, the estimate $\hat{K}^\top \Psi(\bar{x}) \in \mathbb{R}^m$ has to be re-projected to the manifold $M$, see~\cite{mauroy2019koopman,van2023reprojection}.
In \cite{vanGoor2025}, a \emph{geometric reprojection} based on a variance estimate for the lifted dynamics was introduced. 
We build on this probabilistic framework in the sequel and analyse the first and second moments of the prediction residual in the lifted space, which form the basis for the multi-step error characterisation in Section~\ref{sec:one_step_predictions}.

\section{Moment-Based Prediction Error Analysis in the Lifted Space} \label{sec:one_step_predictions}

We now analyse the prediction residual in the lifted space, focusing on its first and second moments as the basis for the multi-step error decomposition. To this end we analyse 
\begin{align*}
    z(t+1) &= A z(t), &
    z(0) &= \Psi(\bar{x})
\end{align*}
for initial states $\bar{x} \in X$.
The quantity to be designed is $A \in \rmspace{m}{m}$, \ie the lift $\Psi$ is fixed.
According to the previous discussion, a reasonable choice is $\hat{A} \coloneqq \hat{K}^\top$, but we carry out the analysis for a generic~$A$, occasionally commenting on the specific case $A = \hat{A}$. Hence, everything defined next will depend on the choice of $A$, but we suppress this dependence in the notation to improve readability.

We denote the trajectories of the lifted system by $z(\cdot,\bar{x}) \colon \mathbb{N}_0 \to \mathbb{R}^m$.
Of special interest is how well $z(\cdot,\bar{x})$ approximates $\Psi(x(\cdot,\bar{x}))$ and its zero-mean property with respect to the empirical measure defined by the samples used in the training process. To this end, we first look at the one-time-step error to gain structural insight, which we then leverage in the analysis of the multi-step and finally in the development of a novel approximation scheme.

Let $\mathcal{O}$ denote the subspace topology on $X \subset \mathbb{R}^n$, and $\mathcal{A} = \sigma(\mathcal{O})$ the corresponding Borel $\sigma$-algebra. 
Additionally, consider the states $x_1, \dots, x_L$ in \eqref{eq:edmd_data}, together with their associated \emph{Dirac measure} $\delta_i \colon \mathcal{A} \to \mathbb{R}$, which is given by
\begin{equation*}
    \delta_i(B) \coloneqq \begin{cases*}
        1 & $x_i \in B$ \\ 
        0 & otherwise
    \end{cases*}
\end{equation*}
Since $\delta_i(X) = 1$, each Dirac measure $\delta_i$ is a probability measure, and thus the \emph{empirical measure}
\begin{equation*}
    P_L \coloneqq \frac{1}{L} \sum_{i=1}^L \delta_i 
\end{equation*}
turns $(X,\mathcal{A},P_L)$ into a probability space. 
If $L$ is large enough, then for each event $B \in \mathcal{A}$ the frequency interpretation of probability suggests that $P_L(B) \simeq P_\star(B)$, i.e., $P_L$ approximately resembles its counterpart~$P_\star$ of the original probability space $(X,\mathcal{A},P_\star)$ used for sampling of initial states $x_i$. 
Thus, when we wish to compute a quantitative property for the probability measure $P_\star$ -- such as expectation or variance -- we may approximate it by computing that property for the empirical measure. 
We apply this philosophy throughout the work. 

For each $\bar{x} \in X$ we have
\begin{equation*}
    (\Psi \circ f)(\bar{x}) = A \Psi(\bar{x}) + \left( (\Psi \circ f)(\bar{x}) - A \Psi(\bar{x}) \right),
\end{equation*}
which motivates the introduction of the \emph{residual map} $R \colon X \to \mathbb{R}^m$ as
\begin{equation} \label{eq:residual}
        R(\bar{x}) \coloneqq (\Psi \circ f)(\bar{x}) - A \Psi(\bar{x}) .
\end{equation}
This describes the error that we make in the lifted space for one time step.
It is thus desirable to choose $A$ such that the residual map is \enquote{small} in some sense, and we will leverage the introduced notions to enable a probabilistic interpretation of the lifted dynamics and, then, define a \enquote{small}.

Since $R$ is continuous with respect to the subspace topo\-logy, it is necessarily $\mathcal{A}$-measurable.
This means that the residual map can be considered a random variable with values in $\mathbb{R}^m$. 
We will analyse its first and second moments in the following.

First, we investigate the residual's expectation.
It will be convenient to introduce for $t \in \{0, \dots, N\}$ the arithmetic means
\begin{equation}\label{eq:means:arithmetic}
    \nu_t \coloneqq \frac{1}{L} \sum_{i=1}^L \Psi(x(t,x_i)) .
\end{equation}
of the collected lifted states at time $t$.
We can state the following immediate consequence.
\begin{lemma} \label{lemma:residual_expectation_vanish}
    The first moment $\exv{R}$ of the residual vanishes if and only if $\nu_1 = A \nu_0$. 
\end{lemma}
\begin{proof}
    We can express the expectation as
    \begin{align*}
        \exv{R} &= \int R \, \mathrm{d} P_L = \frac{1}{L} \sum_{i=1}^L R(x_i) 
        = \frac{1}{L}\sum_{i=1}^L \Big[\Psi(y_i) - A \Psi(x_i)\Big] \\
        &= \frac{1}{L} \left( \sum_{i=1}^L \Psi(y_i) - A \sum_{i=1}^L \Psi(x_i) \right) = \nu_1 - A \nu_0 , 
    \end{align*}
    which directly implies the assertion.
\end{proof}
Let us denote the subset of $\rmspace{m}{m}$ of matrices $A$ that lead to $\exv{R} = 0$ by $E_0$, \ie we have
\begin{equation}\label{eq:matrix:E0}
    E_0 = \{ A \in \rmspace{m}{m} \mid \nu_1 = A \nu_0 \} .
\end{equation}
As we establish next, under a mild and standard assumption on the dictionary, the predictor $\hat{A} = \hat{K}^\top$ obtained from the EDMD algorithm belongs to $E_0$, which has already been noted in~\cite{vanGoor2025}.
\begin{proposition}\label{prop:edmd_residual_expectation}
    Suppose that $\psi_j \equiv 1$ for some member $\psi_j$ of the chosen basis $(\psi_1, \dots, \psi_m)$ for $\linspace{V}$.
    Then, $\hat{A} \in E_0$. 
\end{proposition}
\begin{proof}
    Let $\mathbf{1} \in \mathbb{R}^L$ denote the vector of ones. 
    By definition of the EDMD algorithm, $\hat{A} = G_1 G_0^\top (G_0 G_0^\top)^{-1}$, which gives
    \begin{equation*}
        \hat{A} \nu_0 = \frac{1}{L} \hat{A} G_0 \mathbf{1} = \frac{1}{L} G_1 G_0^\top (G_0 G_0^\top)^{-1} G_0 \mathbf{1} .
    \end{equation*}
    Hence, if we can show $G_0^\top (G_0 G_0^\top)^{-1} G_0 \mathbf{1} = \mathbf{1}$, the assertion follows from \lemmaref{lemma:residual_expectation_vanish}.
    
    Without loss of generality, assume that $\psi_m$ is the constant observable, \ie $\psi_m \equiv 1$. 
    Then, the data matrix has the form
    \begin{equation*}
        G_0 = \Matrix{\hat{G} \\ \mathbf{1}^\top},
    \end{equation*}
    where $\hat{G}$ denotes the first $m-1$ rows of $G_0$.
    Therefore, using the Schur complement
    \begin{equation*}
        S_\mathrm{c} = L - \mathbf{1}^\top \hat{G}^\top (\hat{G} \hat{G}^\top)^{-1} \hat{G} \mathbf{1} 
    \end{equation*}
    of $\hat{G} \hat{G}^\top$ in $G_0 G_0^\top$, the inverse of $G_0 G_0^\top$ can be written as $(G_0 G_0^\top)^{-1} = V^\top H V$, with
    \begin{align*}
        V &= \Matrix{
            I & 0 \\
            -\mathbf{1}^\top \hat{G}^\top (\hat{G} \hat{G}^\top)^{-1} & 1 
        }, &
        H &= \Matrix{
            (\hat{G} \hat{G}^\top)^{-1} \\ & S_\mathrm{c}^{-1}
        },
    \end{align*}
    Then, straightforward calculations reveal
    \begin{align*}
        &V G_0 \mathbf{1} = \Matrix{\hat{G} \mathbf{1} \\ S_\mathrm{c}}, &
        H V G_0 \mathbf{1} &= \Matrix{(\hat{G} \hat{G}^\top)^{-1} \hat{G} \mathbf{1} \\ 1}, \\
        &V^\top H V G_0 \mathbf{1} = \Matrix{\mathbf{0} \\ 1} ,
    \end{align*}
    which implies $G_0^\top (G_0 G_0^\top)^{-1} G_0 \mathbf{1} = \mathbf{1}$, as required.
\end{proof}

In conclusion, the finite-dictionary and finite-data Koopman compression approximates the true lifted dynamics in the sense of matching the first and second moments of the prediction residual. Under the assumption that the residual has zero expectation, it is possible to derive, for a general $A$, a closed-form expression for the variance of the residual, which depends quadratically on~$A$.
\begin{corollary} \label{coro:residual_variance}
    Let $A \in E_0$ be arbitrary.
    Then, the second moment $\Sigma \coloneqq \var{R}$ of $R$ satisfies
    \begin{equation*}
        \Sigma = \frac{1}{L} \left( G_1 - A G_0 \right) \left( G_1 - A G_0 \right)^\top .
    \end{equation*}
    Additionally, the trace of the variance is given by
    \begin{equation*}
        \tr(\Sigma) = \frac{1}{L} \| G_1 - A G_0 \|_\mathrm{F}^2 .
    \end{equation*}
\end{corollary}
\begin{proof}
    Using the assumption that $A \in E_0$ and thus $\exv{R} = 0$, the formula for the variance follows directly from the computation
    \begin{align*}
        \Sigma &= \exv{\left( R - \exv{R} \right) \left( R - \exv{R} \right)^\top}
        = \exv{R R^\top} \\
        &= \frac{1}{L} \sum_{i=0}^L \left( \Psi(y_i) - A \Psi(x_i) \right) \left( \Psi(y_i) - A \Psi(x_i) \right)^\top \\
        &= \frac{1}{L} ( \sum_{i=0}^L \Psi(x_i) \Psi(y_i)^\top - \sum_{i=0}^L \Psi(y_i) \Psi(x_i)^\top A^\top \\
        &\quad- \sum_{i=0}^L A \Psi(x_i) \Psi(y_i)^\top + \sum_{i=0}^L A \Psi(x_i) \Psi(x_i)^\top A^\top ) \\
        &= \frac{1}{L} ( G_1 G_1^\top - G_1 G_0^\top A^\top - A G_0 G_1^\top 
        + A G_0 G_0^\top A^\top ) ,
    \end{align*}
    where we have used the definition of $G_0$ and $G_1$. 
   Now, the claimed expression follows directly from the cyclic property of the trace operator.
\end{proof}

\section{Multi-Step Predictions and Prediction Errors} \label{sec:multi_step_predictions}

In light of this result, and under the assumption that the basis for $\linspace{V}$ contains a constant observable, the EDMD algorithm can be interpreted as selecting the element of $E_0$ that minimises $\tr(\var{R})$. Note that this interpretation hinges on the fact that $\hat{A}$ lies in $E_0$; otherwise, the expression for the variance changes. Moreover, this result does not directly transfer to multi-step predictions. We will address this limitation next.

We consider, for $t \in \{1,\dots,N\}$, the family of random variables $R_t \colon X \to \mathbb{R}^m$, where $R_t$ is given by
\begin{equation} \label{eq:multi_step_residual}
    R_t(\bar{x}) \coloneqq \Psi(x(t, \bar{x})) - z(t, \bar{x}) ,
\end{equation}
describing the prediction error in the lifted space at time $t$. 
Our goal is to find $A$ such that, for each $t \in \{1, \dots, N\}$,
\begin{enumerate}
    \item[(i)] 
        the first moment $\exv{R_t}$ of the residual $R_t$ vanishes at time $t$. 
    \item[(ii)]
        The second moment $\var{R_t}$ of the residual $R_t$ at time $t$ is minimised. 
\end{enumerate}
Consider the expectation 
\begin{equation*}
    \mu_t \coloneqq \exv{R_t} .
\end{equation*}
To analyse how $\mu_t$ evolves over time, it is convenient to introduce, for $t \in \{0, \dots,N-1\}$, the family of random variables $D_t \colon X \to \mathbb{R}^m$ by
\begin{equation}\label{eq:Dt}
    D_t(\bar{x}) \coloneqq \Psi(x(t+1,\bar{x})) - A \Psi(x(t,\bar{x})) = R(x(t,\bar{x})) ,
\end{equation}
which describe the \emph{local error} at time $t$.
While the residual~$R_t$ describes the overall error at time~$t$, the local error describes the newly added error that would be added at time $t$ if the lifted system was initialised at the correct state $\Psi(x(t,\bar{x}))$. 
In particular, $D_0 = R_1 = R$ holds, where the residual map $R$ is defined in \eqref{eq:residual}.

The following propostion shows how $\mu_{t+1}$ can be described as a function of $\mu_t$ and the expectation of the local error,
\begin{equation*}
    \beta_t \coloneqq \exv{D_t} .
\end{equation*}
This allows us to analyse how the expectation of the residual evolves.
\begin{proposition} \label{prop:multi_step_residual}
    For each $t \in \{1, \dots,N-1\}$, the expectation $\mu_t$ of the residual $R_t$ obeys
    \begin{equation*}
        \mu_{t+1} = A \mu_t + \beta_t .
    \end{equation*}
\end{proposition}
\begin{proof}
    Fix any $t \in \{1,\dots,N-1\}$ and let $\bar{x} \in X$ be arbitrary. 
    Note that we have $x(t+1,\bar{x}) = f(x(t,\bar{x}))$ and $z(t+1,\bar{x}) = A z(t,\bar{x})$ by definition.
    Hence, by definition of the residuals, we obtain
    \begin{align*}
        R_{t+1}(\bar{x}) &= \Psi(x(t+1),\bar{x}) - z(t+1,\bar{x}) \\
        &= (\Psi \circ f)(x(t,\bar{x})) - A z(t,\bar{x}) \\
        &\overset{\eqref{eq:multi_step_residual}}{=} (\Psi \circ f)(x(t,\bar{x})) - A \left( \Psi(x(t,\bar{x})) - R_t(\bar{x}) \right) \\
        &= \left( (\Psi \circ f)(x(t,\bar{x})) - A  \Psi(x(t,\bar{x})) \right) + A R_t(\bar{x}) \\
        &\overset{\eqref{eq:residual}}{=} R(x(t,\bar{x})) + A R_t(\bar{x}) 
        = A R_t(\bar{x}) + D_t(\bar{x}).
    \end{align*}
    The result follows from linearity of the expectation.
\end{proof}

This result has an interesting interpretation:
the expected residual follows a linear dynamic given by the chosen predictor $A$ with an additional \enquote{disturbance} $\beta_t$ at time~$t$. 
This allows for the application of techniques from linear system theory to analyse the expected residual error. 

\begin{corollary} \label{coro:expectation_residual_multi_vanish}
    Given any choice of $A \in \rmspace{m}{m}$, then
    \begin{align*}
        \mu_1 = \dots = \mu_N &= 0 &
        &\iff &
        \beta_0 = \dots = \beta_{N-1} &= 0 
    \end{align*}
\end{corollary}
\begin{proof}
    The proof follows essentially from observing that $R_1 = D_0$, which gives $\mu_1 = \beta_0$. 
    Then, the solution to the linear system $\mu_{t+1} = A \mu_t + \beta_t$ yields for each $t \in \{1, \dots, N\}$
    \begin{equation*}
        \mu_t = A^{t-1} \mu_1 + \sum_{k=1}^{t-1} A^{t-1-k} \beta_k = \sum_{k=0}^{t-1} A^{t-1-k} \beta_k ,
    \end{equation*}
    from which we can immediately see that $\beta_0 = \dots = \beta_{N-1} = 0$ implies $\mu_1 = \dots = \mu_N = 0$. 

    For the reverse direction, note that for $0 = \mu_1 = \beta_0$, the expression for $\mu_t$ implies for each $t \in \{1,\dots,N-1\}$
    \begin{equation*}
        0 = \mu_{t+1} = \beta_{t} ,
    \end{equation*}
    as claimed.
\end{proof}

Corollary~\ref{coro:expectation_residual_multi_vanish} shows that the \emph{overall} residuals vanish in expectation for \emph{all times} $t \in \{1,\ldots,N\}$ if and only if the \emph{local errors} vanish in expectation. 
Thanks to the linear dependence of $\beta_t$ on $A$, we may represent this relationship by linear conditions on~$A$. 
The following Corollary extends \lemmaref{lemma:residual_expectation_vanish} to multiple steps.

\begin{corollary} \label{coro:expectation_local_error_vanish}
    For each $t \in \{0, \dots, N-1\}$, the first moment $\beta_t = \exv{D_t}$ of the local error at time $t$ vanishes if and only if $\nu_{t+1} = A \nu_t$.
\end{corollary}
\begin{proof}
    Fix any $t \in \{0, \dots, N-1\}$.
    Then
    \begin{equation*}
        \beta_t = \frac{1}{L} \sum_{i=1}^L \Psi(x(t+1,x_i)) - A \Psi(x(t,x_i)) 
        = \nu_{t+1} - A \nu_t ,
    \end{equation*}
    and the assertion follows.
\end{proof}

We thus introduce for each $t \in \{0, \dots,N-1\}$ the set
\begin{equation*}
    E_t = \{ A \in \rmspace{m}{m} \mid \forall k \in \{0, \dots, t\} \colon \nu_{k+1} = A \nu_k \} ,
\end{equation*}
so that $\beta_0 = \dots = \beta_t = 0$ if and only if $A \in E_t$.
Note that this definition includes our previously defined $E_0$, and we have the inclusions $E_{t+1} \subseteq E_t$. 
Additionally, $A \in E_{t-1}$ is equivalent to $\mu_1 = \dots = \mu_t = 0$.

This completes the theoretical analysis of unbiasedness conditions.

\section{Parametrisation and Degrees of Freedom}
Having characterised the conditions under which the multi-step residual is unbiased, we now turn to the constructive aspect of the problem. In this section, we parametrise the set $E_{N-1}$, identify the degrees of freedom available in choosing $A$, and exploit them to design a predictor that enforces unbiasedness while reducing variance.
\subsection{Degrees of Freedom}
To do so, we define the matrices
\begin{align*}
    B &= \Matrix{ \nu_0 & \dots & \nu_{N-1} } , &
    C &= \Matrix{ \nu_1 & \dots & \nu_{N} }.
\end{align*}
Assume that $m \geq N$ and take the QR decomposition\footnote{We assume here that $B$ has full rank for simplicity, but this can be generalised.}
\begin{equation*}
    B = Q \bar{R} = \Matrix{Q_1 & Q_2} \Matrix{\hat{R} \\ 0} = Q_1 \hat{R},
\end{equation*}
with $\hat{R} \in \rmspace{N}{N}$. 
\begin{proposition} \label{prop:zero_expectation_predictor}
    Assume that $m \geq N$.
    Then $\mu_1 = \dots = \mu_N = 0$ if and only if $A = C \hat{R}^{-1} Q_1^\top + A_2 Q_2^\top$ for some $A_2 \in \rmspace{m}{(m-N)}$.
\end{proposition}
\begin{proof}
    By combining \cororef{coro:expectation_residual_multi_vanish} and \cororef{coro:expectation_local_error_vanish} and using the newly defined matrices $B$ and $C$, we see that $\mu_1 = \dots = \mu_N = 0$ is equivalent to
    \begin{align*}
        A &\in E_{N-1} &
        &\iff &
        A B = C.
    \end{align*}
    We thus seek to parametrise the matrices $A$ that satisfy $A B = C$.
    Since $Q$ is orthogonal, we can decompose any $A \in \rmspace{m}{m}$ as $A = A_1 Q_1^\top + A_2 Q_2^\top$, by setting $A_1 = A Q_1$ and $A_2 = A Q_2$.
    This decomposition gives
    \begin{equation*}
        A B = A_1 Q_1^\top Q_1 \hat{R} + A_2 Q_2^\top Q_1 \hat{R} = A_1 \hat{R} ,
    \end{equation*}
    and hence $A B = C$ if and only if $A = A_1 Q_1^\top + A_2 Q_2^\top$ with $A_1 = C \hat{R}^{-1}$.
\end{proof}

The decomposition $A = A_1 Q_1^\top + A_2 Q_2^\top$ reveals the degrees of freedom available in selecting $A \in E_{N-1}$.
In this basis, one \emph{must} choose $A_1 = C \hat{R}^{-1}$, and the remaining degrees of freedom are given by $A_2$.

\subsection{Minimising Variance}

The analysis of the previous subsection has led to  numerically tractable conditions to ensure $A \in E_{N-1}$. 
In particular, the first moment of the residual $\exv{R_t}$ vanishes.
Next, we want to leverage the remaining degrees of freedom to address our second design criterion, namely, to minimise the second moment $\var{R_t} \eqqcolon \Sigma_t$.

We have seen in the beginning that the EDMD predictor achieves this for $N=1$, and we want to generalise this procedure now. 
In particular, we want to derive a closed-form solution that can, then, be algorithmically exploited.

To study the variance $\Sigma_t$ at a time $t$, we begin by describing its evolution over time if $A \in E_{N-1}$. 
Since $R_1 = R$, the expression for $\Sigma_1$ is given by \cororef{coro:residual_variance}.
For the future time steps, we rely on the local error $D_t$. 
For this, denote by
\begin{align*}
    \Omega_t &\coloneqq \var{D_t}, &
    \Gamma_t &\coloneqq \exv{R_t D_t^\top},
\end{align*}
the variance of $D_t$ and the covariance of $R_t$ and $D_t$, respectively. 
The variance of the local error admits a simple form.
\begin{corollary} \label{coro:local_error_variance}
    If $A \in E_{N-1}$, then for each $t \in \{0,\dots,N-1\}$ the variance $\Omega_t$ of $D_t$ is given by
    \begin{align*}
        \Omega_t = \frac{1}{L} (G_{t+1} - A G_t) (G_{t+1} - A G_t)^\top,
    \end{align*}
    with trace
    \begin{equation*}
        \tr(\Omega_t) = \frac{1}{L} \| G_{t+1} - A G_t \|_\mathrm{F}^2
    \end{equation*}
\end{corollary}
\begin{proof}
    The assumption implies $\var{D_t} = \exv{D_t D_t^\top}$, and so the expression and trace for $\Omega_t$ can be proven with the same steps as in \cororef{coro:residual_variance}.
\end{proof}
Similar to the evolution of the first moment, the second moment of $R_t$ can also be expressed in terms of its previous variance and the one-step variance. 
However, a covariance term has to be added.
\begin{proposition} \label{prop:residual_variance_multi_step}
    Let $A \in E_{N-1}$. 
    Then,
    \begin{align*}
        \Sigma_{t+1} &= A \Sigma_t A^\top + A \Gamma_t + \Gamma_t^\top A^\top + \Omega_t ,
    \end{align*}
    for each $t \in \{1,\dots,N-1\}$. 
\end{proposition}
\begin{proof}
    Fix any $t \in \{1,\dots,N-1\}$.
    By \cororef{coro:expectation_residual_multi_vanish}, the assumption $A \in E_{N-1}$ in particular implies $\mu_{t+1} = 0$, $\mu_t = 0$ and $\beta_t = 0$. 
    Hence, the variance is computed as $\Sigma_{t+1} = \exv{R_{t+1} R_{t+1}^\top}$. 
    Moreover, from the proof of \propref{prop:multi_step_residual} we have $R_{t+1} = A R_t + D_t$. 
    Combing this implies
    \begin{align*}
        \Sigma_{t+1} &= \exv{(A R_t + D_t)(A R_t + D_t)^\top} \\
        &= A \exv{R_t R_t^\top} A^\top + A \exv{R_t D_t^\top} + \exv{D_t R_t^\top} A^\top \\
        &\quad+ \exv{D_t D_t^\top} .
    \end{align*}
    where we have used the linearity of expectation. 
    Since $\mu_t = \beta_t = 0$, we have $\exv{R_t R_t^\top} = \var{R_t}$ and $\exv{D_t D_t^\top} = \var{D_t}$, which establishes the expression for $\Sigma_{t+1}$.
\end{proof}
Hence, the variance of $\Sigma_{t+1}$ consists of three parts:
\begin{enumerate}
    \item[(i)] the forward-propagated variance of the last step; given by $A \Sigma_t A^\top$. 
    \item[(ii)] the covariance of $A R_1$ and $D_1$; given by $A \Gamma_1 + \Gamma^\top A^\top$.
    \item[(iii)] the variance of $D_t$; given by $\Omega_t$.
\end{enumerate}
Out of these terms, $\Omega$ is the only one that is quadratic in $A$, the other two exhibit a more complicated dependence of $A$, which makes an exact analysis of $\Sigma_{t+1}$ cumbersome. 
To overcome this, we propose an approximation of the true variance that results in a simplified analysis. 

To this end, we note the following idealised result.
\begin{lemma} \label{lemma:zero_local_variance}
    Let $A \in E_{N-1}$ and $t \in \{1, \dots,N\}$. 
    If
    \begin{equation*}
        \forall k \in \{0, \dots, t-2\} \colon \tr(\Omega_k) = 0
    \end{equation*}
    holds, then $\Sigma_{t} = \Omega_{t-1}$.
\end{lemma}
\begin{proof}
    Fix any $t \in \{1, \dots, N\}$. 
    Since $A \in E_{N-1}$, the variance $\Sigma_{t}$ is computed as
    \begin{align*}
        \Sigma_t &= \exv{R_{t} R_{t}^\top} 
        = \frac{1}{L} \sum_{i=1}^L R_{t}(x_i) R_{t}(x_i)^\top \\
        &= \left( G_{t} - A^{t} G_0 \right) \left( G_{t} - A^{t} G_0 \right)^\top .
    \end{align*}
    By \cororef{coro:local_error_variance}, the made assumption implies
    \begin{align*}
        \forall k \in \{0, \dots, t-2\} \colon G_{k+1} = A G_k ,
    \end{align*}
    from which we conclude
    \begin{equation*}
        A^t G_0 = A^{t-1} A G_0 = A^{t-2} A G_1 = \dots = A G_{t-1} .
    \end{equation*}
    The assertion then follows directly from inserting this equality in the expression for $\Sigma_t$.
\end{proof}
\begin{remark}
    In the light of \propref{prop:residual_variance_multi_step}, for $t>1$ the assumption in \lemmaref{lemma:zero_local_variance} is equivalent to
    \begin{equation*}
        A \Sigma_{t-1} A^\top + A \Gamma_{t-1} + \Gamma_{t-1}^\top A^\top = 0 .
    \end{equation*}
    Elaborating on this would be out of scope here, but a refined analysis of the relation between and the evolution of the propagated variance $A \Sigma_t A^\top$ and the covariance $A \Gamma_t + \Gamma_t^\top A^\top$ is an interesting question for future research.
\end{remark}
In principle, it is thus possible to minimise $\tr(\Sigma_N)$ by minimising $\tr(\Omega_{N-1})$ under the constraint that $\tr(\Omega_t) = 0$ for $t < N-1$.
In case of the variance, however, achieving this constraint even for a single time step is in general not possible, since $A$ does not offer enough degrees of freedom. 
Hence, we cannot expect to \emph{exactly} satisfy the assumption in \lemmaref{lemma:zero_local_variance} for any $t > 0$.

However, based on \lemmaref{lemma:zero_local_variance}, we may still argue that as long as $\tr(\Omega_k)$ is \enquote{reasonably} small for $k \in \{0, \dots,t-2\}$, then the approximation $\Sigma_t \approx \Omega_{t-1}$ is fair. 
We thus conclude that a possible approach to obtain $A$ is to simultaneously minimise
\begin{align*}
    \frac{1}{L} &\| G_{1} - A G_0 \|^2 & &\dots &
    \frac{1}{L} &\| G_N - A G_{N-1} \|^2
\end{align*}
under the constraint $A \in E_{N-1}$. 
This can be seen as (approximately) minimising $\tr(\Sigma_1), \dots, \tr(\Sigma_N)$. 

\section{Design of the Predictor} \label{sec:predictor}

We are finally prepared to introduce a design procedure for $A$. 
As a first requirement, we demand $A \in E_{N-1}$, which can be easily satisfied for $m>N$ by virtue of \propref{prop:zero_expectation_predictor}: we set $A = A_1 Q_1^\top + A_2 Q_2^\top$ with $A_1 = C \hat{R}^{-1}$. 

The second requirement is to minimise $\tr(\Sigma_N)$, which we base on the previous discussion. 
In particular, we search for an $A_2$ such that for all $t \in \{1, \dots, N-1\}$
\begin{equation*}
    \frac{1}{L} \| G_{t+1} - A G_t \|_\mathrm{F}^2 \approx 0 .
\end{equation*}
This ensures that the approximation $\Sigma_t \approx \Omega_{t-1}$ is reasonable, and thus at the same time it readily minimises $\tr(\Sigma_t) \approx \tr(\Omega_{t-1})$.

However, this requirement on $A_2$ is not straightforward to achieve, since the individual objectives $\| G_{t+1} - A G_t \|_\mathrm{F}^2$ compete against each other. 
In this work, we propose to select $A$ that minimises the simple additive objective
\begin{equation} \label{eq:additive_objective}
    \sum_{t=0}^{N-1} \| G_{t+1} - A G_t \|_F^2 
\end{equation}
over $E_{N-1}$.
By virtue of the explicit parametrisation of $A$, this can be cast as an unconstrained least-squares problem. 
Introducing the aggregated data matrices
\begin{align} \label{eq:aggregated_data_matrices}
    G &= \Matrix{ G_0 & \dots & G_{N-1} }, &
    S &= \Matrix{ G_1 & \dots & G_{N} },
\end{align}
turns the objective \eqref{eq:additive_objective} into
\begin{equation*}
    \| S - A G \|_\mathrm{F}^2 
    = \| (S - A_1 Q_1^\top G) - A_2Q_2 G \|_\mathrm{F}^2 
    \eqqcolon \| \tilde{S} - A_2 \tilde{G} \|_\mathrm{F}^2 .
\end{equation*}
Assuming that $\tilde{G}$ has full row rank, the unique least-squares solution is achieved by $A_2 = \tilde{S} \tilde{G}^\top (\tilde{G} \tilde{G}^\top)^{-1}$, from which we can recover our \emph{proposed predictor}
\begin{equation} \label{eq:predictor}
    A = C \hat{R}^{-1} Q_1^\top + \tilde{S} \tilde{G}^\top (\tilde{G} \tilde{G}^\top)^{-1} Q_2^\top .
\end{equation}

While conceptually simple, this procedure has several advantages.
First, it guarantees $\mu_1 = \dots = \mu_N = 0$, ensuring that the predictor is unbiased over the entire prediction horizon.
Second, the resulting unbiasedness allows the variances to be computed without having to account for the expectations.
Third, the predictor is obtained by solving a convex optimisation problem, guaranteeing global optimality.
Finally, the optimal value of the least-squares problem provides a (rough) upper bound for each $\tr(\Omega_t)$; if this bound is too large, the prediction horizon $N$ can be reduced accordingly.

It is important to note that the proposed method is not equivalent to applying the EDMD algorithm to the aggregated data matrices $G$ and $S$.
In EDMD, the predictor would be given by $\bar{A} = S G^\top (G G^\top)$, which does \emph{not} lead to $\mu_1 = \dots = \mu_N = 0$. 

\section{Numerical Example} \label{sec:example}

We consider the unforced and undamped Duffing oscillator $\ddot{x}(t) = x(t) - x(t)^3$. 
We use the state-space representation $x_1 = x$ and $x_2 = \dot{x}$, and sample the continuous dynamics with sampling time $\Delta t = 0.1$. 
We draw $L = 10,000$ states uniformly from $X = [-1,1]^2$, from which we simulate $N=5$ steps forward in time.
The observables $(\psi_1,\dots,\psi_m)$ correspond to the monomial up to degree ten, which gives $m=66$. 
This setup defines the data matrices $G_0, \dots, G_5$.

We consider three predictors:
\begin{enumerate}
    \item[(i)] 
        The original EDMD predictor based on $G_1$ and $G_2$, \ie $\hat{A} = G_1 G_0^\top (G_0 G_0^\top)^{-1}$.
    \item[(ii)] 
        The EDMD predictor that uses the aggregated data matrices in \eqref{eq:aggregated_data_matrices}, \ie $\bar{A} = S G^\top (G G^\top)^{-1}$.
    \item[(iii)] 
        Our proposed predictor $A$ in \eqref{eq:predictor}.
\end{enumerate}
First, consider \figref{fig:residual_expected}. 
It shows how far the expected residual -- measured by the empirical measure -- is away from zero. 
We can see that our predictor is indeed unbiased, as desired.
The one-step EDMD predictor $\hat{A}$ leads to the highest expected residual error. 
This is expected, since no multi-step information were used in the training of $\hat{A}$.
The expectation of the residual stemming from the aggregated predictor $\bar{A}$ -- which does have access to multi-step information -- is smaller, but the predictor is still biased. 
This highlights the importance of the analysis and synthesis performed in this work.
\begin{figure}
    \centering
    \begin{tikzpicture}[scale=0.75]
        \begin{axis}[%
            xmin=1,
            xmax=5,
            xlabel=$t$,
            ylabel=$\| \mu_t \|$,
            ymin=-0,
            ymax=1e-2,
            grid=both,
            legend pos=north west,
        ]
            \addplot[
                color=myblue,
                mark=square,
                ]
                coordinates {
                    (1,0)
                    (2,1.5132e-04)
                    (3,0.0012)
                    (4,0.0062)
                    (5,0.0231)
                }; 
            \addlegendentry{\footnotesize{$\hat{A}$}}
            \addplot[
                color=mygreen,
                mark=square,
                ]
                coordinates {
                    (1,6.9281e-05)
                    (2,1.6042e-04)
                    (3,3.1498e-04)
                    (4,6.5528e-04)
                    (5,0.0014)
                }; 
            \addlegendentry{\footnotesize{$\bar{A}$}}
            \addplot[
                color=myorange,
                mark=square,
                ]
                coordinates {
                    (1,0)
                    (2,0)
                    (3,0)
                    (4,0)
                    (5,0)
                }; 
            \addlegendentry{\footnotesize{$A$}}
        \end{axis}
    \end{tikzpicture}
    \caption{Evolution of the residual expectation over time 
for the three predictors (\(\hat{A}\): one-step EDMD, 
\(\bar{A}\): aggregated EDMD, \(A\): proposed unbiased predictor).}
    \label{fig:residual_expected}
\end{figure}
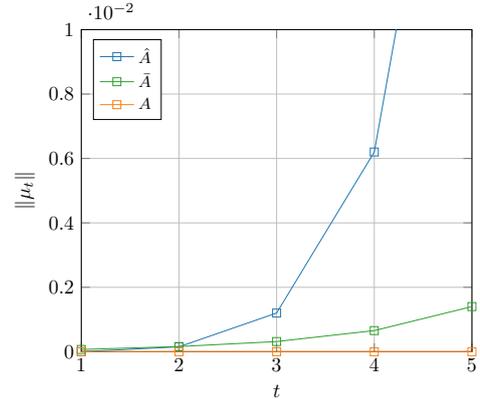

Turning to \figref{fig:residual_variance}, we can see how the variance associated with the three predictors evolves over time. 
As before, the predictor $\hat{A}$ performs worst, which is expected. 
The key point to note is that the proposed predictor $A$ matches the variances of $\bar{A}$. 
Therefore, the associated residuals vary by the same amount around their respective expectations.
The predictions based on $A$, however, are more interpretable, since the residual varies around zero, whereas the predictions based on $\bar{A}$ are biased. 
Hence, given that the dictionary -- and thus the degrees of freedom in $A$ -- are large enough, we do not have to pay for the unbiased $A$ by a higher variance.

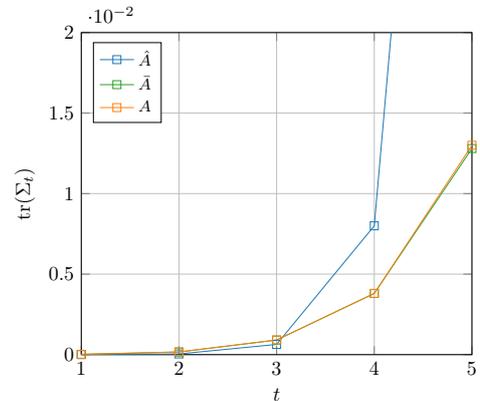
\begin{figure}
    \centering
    \begin{tikzpicture}[scale=0.75]
        \begin{axis}[%
            xmin=1,
            xmax=5,
            xlabel=$t$,
            ylabel=$\tr(\Sigma_t)$,
            ymin=-0,
            ymax=2e-2,
            grid=both,
            legend pos=north west,
        ]
            \addplot[
                color=myblue,
                mark=square,
                ]
                coordinates {
                    (1,3.0104e-06)
                    (2,4.3350e-05)
                    (3,6.3559e-04)
                    (4,0.0080)
                    (5,0.0773)
                }; 
            \addlegendentry{\footnotesize{$\hat{A}$}}
            \addplot[
                color=mygreen,
                mark=square,
                ]
                coordinates {
                    (1,1.9095e-05)
                    (2,1.6558e-04)
                    (3,9.0481e-04)
                    (4,0.0038)
                    (5,0.0128)
                }; 
            \addlegendentry{\footnotesize{$\bar{A}$}}
            \addplot[
                color=myorange,
                mark=square,
                ]
                coordinates {
                    (1,1.9366e-05)
                    (2,1.6709e-04)
                    (3,9.1227e-04)
                    (4,0.0038)
                    (5,0.0130)
                }; 
            \addlegendentry{\footnotesize{$A$}}
        \end{axis}
    \end{tikzpicture}
    \caption{Evolution of the residual variance over time 
for the three predictors (\(\hat{A}\): one-step EDMD, 
\(\bar{A}\): aggregated EDMD, \(A\): proposed unbiased predictor).}
    \label{fig:residual_variance}
\end{figure}

\section{Conclusion} \label{sec:conclusion}

Accurate long-horizon prediction is essential for data-driven modelling and control, yet Koopman-based predictors  often suffer from structural bias that accumulates over time. This highlights the need for principled methods that enforce unbiasedness while controlling variance.

A multi-step predictor is constructed that guarantees unbiasedness and achieves low variance in the Koopman-lifted space. A probabilistic characterisation of the residual error provides the structural conditions under which unbiased multi-step prediction is possible and exposes the degrees of freedom available for predictor design. These degrees of freedom are then exploited to enforce unbiasedness and to minimise the multi-step variance through a convex optimisation procedure.

A numerical study confirms the effectiveness of the resulting predictor. When the lifted space is sufficiently expressive, the method achieves variance levels comparable to existing EDMD-based approaches while additionally ensuring unbiased behaviour across the entire prediction horizon.

The proposed framework offers a principled tool for constructing reliable multi-step Koopman predictors from finite data. It lays the foundation for uncertainty-aware forecasting, learning - yet model-based control, where bias-free predictions and consistent variance propagation are important.

\bibliographystyle{ieeetr}
\bibliography{references}

\end{document}